\documentclass[11pt, reqno]{amsart}
\usepackage[english]{babel}
\usepackage[T1]{fontenc}
\usepackage{amsmath}
\usepackage{amsthm}
\usepackage{amsfonts}
\usepackage{amssymb}
\usepackage{lmodern} 
\usepackage{indentfirst}		
\usepackage{microtype} 			
\usepackage[numbers]{natbib}
\usepackage{hyphenat}
\usepackage{pgfplots}  	
\usepackage[utf8]{inputenc}
\usepackage{indentfirst}

\definecolor{blackgreen}{RGB}{0,80,0}

\usepackage[plainpages=false,pdfpagelabels,pdfusetitle,pdfdisplaydoctitle, pdfstartpage=1, pdfstartview={{XYZ null null 1.00}}]{hyperref}
\hypersetup{
	pdftitle={INLSa},    
	pdfauthor={Luccas Campos, Mykael Cardoso and Luiz Gustavo Farah},    
	pdfnewwindow=true,      
	colorlinks=true,      
	linkcolor=blackgreen,         
	citecolor=red}

\usepackage[top=2cm, bottom=2cm, left=2cm, right=2cm]{geometry}
\usepackage{mathtools}
\mathtoolsset{showonlyrefs}

\newcommand{\Real}{\mathbb R}
\newtheorem{thm}{Theorem}[section]
\newtheorem{prop}[thm]{Proposition}
\newtheorem{lemma}[thm]{Lemma}
\newtheorem{claim}[thm]{Claim}
\newtheorem{rem}[thm]{Remark}
\newtheorem{definition}[thm]{Definition}

\numberwithin{equation}{section}

\title{Blow-up for the 3D intercritical inhomogeneous NLS with inverse-square potential}
\author{Luccas Campos, Mykael Cardoso and Luiz Gustavo Farah} 
\date{} 

\begin{document}
\maketitle
	
\begin{abstract}\noindent
In this paper we study the focusing inhomogeneous 3D nonlinear Schr\"odinger equation with inverse-square potential in the mass-supercritical and energy-subcritical regime. We first establish local well-posedness in $\dot{H}_a^{s_c}\cap \dot{H}_a^1$, with $s_c=3/2-(2-b)/2\sigma$. Next, we prove the blow-up of the scaling invariant Lebesgue norm for radial solutions and also, with an additional restriction, in the non-radial case.
\end{abstract}

\section{Introduction}
We consider the initial value problem (IVP) for the inhomogeneous nonlinear Schr\"odinger equation with inverse-square potential $(\mbox{INLS})_a$ 
\begin{equation}
\begin{cases}
i \partial_tu  +\Delta u -\frac{a}{|x|^2}u + |x|^{-b} |u|^{2 \sigma}u = 0, \,\,\, x \in \mathbb{R}^3, \,t>0,\\
u(0) = u_0.
\end{cases}
\label{PVI}
\end{equation}


The (INLS)$_a$ model can be interpreted as an extension to the nonlinear Schr\"odinger equation. For instance, when $a=b=0$ we have the classical nonlinear Schr\"odinger equation, extensively studied over the last decades (see Sulem and Sulem \citep{Sulem}, Bourgain \citep{Bo99}, Cazenave \citep{cazenave}, Linares and Ponce \citep{LiPo15}, Fibich \citep{Fi15}, Tao \cite{TaoBook} and the references therein).
The case $b=0$ and $a\neq 0$ is known as the NLS equation with inverse square potential, denoted by (NLS)$_{a}$,
\begin{equation}\label{NLSa}
i\partial_tu +\Delta u -\frac{a}{|x|^2}u + |u|^{2\sigma} u =0.
\end{equation}

This equation has also been attracting attention in recent years (see for instance, \cite{Burq}, \cite{Murphy}, \cite{Murphy1}, \cite{Okazawa2012}, \cite{Zhang}). The equation \eqref{NLSa} appears in several physical settings, such as in quantum field equations or black hole solutions of the Einstein's equations (see e.g. \cite{MR0397192}).   

Alternatively, if $a=0$ and $b\neq 0$ we have the inhomogeneous NLS equation, denoted by INLS, i.e.,
\begin{equation}\label{INLS}
i\partial_tu+\Delta u +|x|^{-b} |u|^{2\sigma} u =0,
\end{equation}
which also has received substantial attention recently (see e.g. \cite{g_8}, \cite{CARLOS}, \cite{paper2},  \cite{CHL20},\cite{Campos}, \cite{DK21}, \cite{MMZ21}, \cite{M_Proc22}, \cite{AT21}, \cite{AK21}, \cite{KLS21}, \cite{LS21}. It can model, for example, nonlinear optical systems with spatially dependent interactions (see e.g. \cite{belmonte2007lie} and the references therein). In particular, it can be thought of as modeling
inhomogeneities in the medium in which the wave propagates (see for instance \cite{MALOMED1}). The (INLS)$_a$ can thus be seen as a general model for various physical contexts, and it has been already studied in \cite{CG_Zamp_21} and \cite{AJK_DCDS_22}.

Here, we are interested in the following range of the parameters:  

\begin{equation}\label{restrictions}
    a > 0\, , \quad 0<b<\frac32 \quad \text{and} \quad \dfrac{2-b}{3} < \sigma <2-b.
\end{equation}
The first restriction in \eqref{restrictions} is related to the dispersive estimate \eqref{dispersive_decay}. It is necessary to deduce Lemma \ref{Lemma-Str2}, which is used, in turn, to prove our local well-posedness result (Theorem \ref{LWP}). The restriction $a>0$ also ensures, via the sharp Hardy inequality
\begin{align}\label{SHI}
	\||x|^{-1}u\|_{L^2}\leq 2\|\nabla u\|_{L^2},\,\,\,\mbox{for all}\,\,\, u \in H^1(\Real^3). \,\,\, 
\end{align}
the coercivity and boundedness of the linear operator in \eqref{PVI}.

Indeed, consider the quadratic form
$$
Q_a(u)=\int |\nabla u(x)|^2+\frac{a}{|x|^2}|u(x)|^2\,dx
$$
defined in $C^{\infty}_c(\mathbb{R}^3\backslash\{0\})$ and its Friedrichs extension denoted by
$$
\mathcal L_a=-\Delta+\frac{a}{|x|^2}
$$
(see Killip, Miao, Visan, Zhang and Zheng \cite{KMVZZ18} and the references therein for more details).
The restriction $a>0$ (or, more broadly, $a>-{1}/{4}$) guarantees that $\mathcal L_a > 0$ and, defining $\|f\|^2_{\dot{H}^1_a}=\|\mathcal L_a^{\frac12}
 f\|^2_{L^2}=(f,\mathcal L_a f)_{L^2}=Q_a(f)$, we also have
\begin{align}\label{equiv}
	\|f\|_{\dot{H}^1_a}\sim \|\nabla f\|_{L^2}, \,\,\,\mbox{for all} \,\,\,f\in \dot H^1(\Real^3).
\end{align}
This is a particular case of the $L^p$-                                                                                                    based Sobolev spaces adapted to the Schr\"odinger operator with inverse-square potential $\mathcal L_a$ equipped with the norms
$$
\|f\|_{H_a^{s,p}}=\|(1+\mathcal L_a)^{\frac s2}f\|_{L^p}\,\,\, \mbox{and}\,\,\, \|f\|_{\dot{H}_a^{s,p}}=\|\mathcal L_a^{\frac s2}f\|_{L^p}.
$$
We use the usual convention for $p=2$, i.e. $\|f\|_{H_a^{s}}=\|f\|_{H_a^{s,2}}$ and $\|f\|_{\dot{H}_a^{s,2}}=\|f\|_{\dot{H}_a^{s}}$. In \cite{KMVZZ18} the authors studied the relations between these spaces and those defined via $(\Delta)^{s/2}$ and we recall their results in Lemma \ref{EquivHs} below. For instance, when $s>0$ and we have
\begin{align}\label{equiv0}
\|\mathcal L_a^\frac s2 f\|_{L^r}\sim \|D^s f\|_{L^r}, 
\end{align} 
for $0<s<2$ and $1<                                                                                                                                                                                                                                                                                                                                                                                                                                                                                                                                                                                                                                                                                                         r<3/s$. Note that the case $s=1$ is a direct consequence of the sharp Hardy inequality as explained above. The relation \eqref{equiv0} will be one of the key ingredients to deduce the local well-posedness result in Theorem \ref{LWP}.

On the other hand, the second and third conditions in \eqref{restrictions} are related to the local well-posedness of \eqref{PVI}. Indeed, the scale-invariant Sobolev space for the initial data in view of the scale symmetry
$$
\lambda\mapsto u_{\lambda}(x,t):=\lambda^{\frac{2-b}{2\sigma}} u(\lambda x,\lambda^2t),
$$
corresponds to $\dot{H}^{s_c}(\Real^3)$, where $s_c=\frac{3}{2}-\frac{2-b}{2\sigma}$. The equation is called intercritical if $0<s_c<1$, which is equivalent to the third condition in \eqref{restrictions}. The condition $0<b<3/2$ stems from the more general restriction $0<b<\min\{N/2,2\}$, for $u_0 \in \dot H^1(\mathbb R^N)$, and it is exploited in detail by Correia and the first and third authors \cite{CCF_Arxiv_22}.

Next, we recall mass and energy conservation along the flow of the $(\mbox{INLS})_a$ equation. If $u(t)$ is a solution to \eqref{PVI} on some time interval $I\subset \mathbb{R}$, $0\in I$, then for any $t\in I$
\begin{equation}\label{Mass}
M[u(t)] =\int  |u(x,t)|^2\, dx =M[u_0]
\end{equation}
and
\begin{equation}\label{Energy}
E_a[u(t)] =\frac12 \int  |\nabla u(x,t)|^2\, dx +\frac a2\int |x|^{-2}|u(x,t)|^2\,dx - \frac{1}{2\sigma+2}\int |x|^{-b}|u(x,t)|^{2\sigma+2}\,dx=E_a[u_0]. 
\end{equation}

Recently, Guzmán and the first author \cite{LUCCASCARLOS} obtained the local well-posedness for the $(\mbox{INLS})_a$ equation in $H^1(\Real^N)$ via Kato's method, which is based on the contraction mapping principle and the Strichartz
estimates. More precisely, they considered $N\geq 3$ and showed that for initial data $u_0\in H^1_a(\Real^N)$ and 
$$
\left\{
\begin{array}{ll}
a>-\frac{(N-2)^2}{4},& \mbox{ if }\frac{1-b}{N}<\sigma\leq \frac{1-b}{N-2}\mbox{ and }0<b<1\\
a>-\frac{(N-2)^2}{4}+\left(\frac{\sigma(N-2)-(1-b)}{2\sigma+1}\right)^2,& \mbox{if} \max\{0,\frac{1-b}{N}\}<\sigma<\frac{2-b}{N-2}
\end{array}
\right.
$$
there exists a $T=T(\|u_0\|_{\dot H^1_a}, N, \sigma, b)$ and a unique solution of \eqref{PVI} satisfying
\begin{align}
	u\in C\left([0,T];H^1_a(\Real^N)\right)\cap L^q\left([0,T]; H^{1,r}_{a}(\Real^N)\right)
\end{align}
where $(q,r)$ is a $L^2$-admissible pair. \\

They also proved the existence of a ground state and a dichotomy between global existence and blow up below the ground state threshold. Furthermore, they proved small data scattering, constructed wave operators in $H^1_a$ and proved that suitable space-time bounds imply scattering in the intercritical case. Later on, An, Jan and Kim 
\cite{AJK_DCDS_22} extended the results of Duyckaerts and Roudenko \cite{DR_Going} for the NLS and of the first and second authors \cite{CC_Going} for the INLS to include a description of the global behavior of solutions to (INLS)$_a$ with mass-energy possibly greater than or equal to the ground state threshold.



We are interested here in studying blow-up rates for (INLS)$_a$. To that end, we first obtain the local theory in $\dot{H_a^{s_c}}(\Real^3)\cap \dot H_a^1(\Real^3)$. We exclude the $L^2$ norm from the well-posedness theorem since it is subcritical with respect to the scaling of the equation. Therefore, the procedure adopted in the proof of Theorem \ref{scteo2} produces a sequence of functions whose mass is unbounded. In order to be able to obtain uniform bounds, we work with a Morrey-Campanato type semi-norm $\rho$, defined by \eqref{defrho}, which somehow mimics the mass and is defined for functions merely in $\dot{H}^{s_c}$. Moreover, the control of this quantity allows us to show that global solutions must have positive energy in the intercritical case (see Theorem \ref{scteo1} below). Next, we state our local theory result.

\begin{thm}\label{LWP} Let  $a>0$, $0<b<\frac{3}{2}$ and $\frac{2-b}{3}<\sigma<2-b$.  If $u_0 \in \dot{H_a^{s_c}}(\Real^3)\cap \dot H_a^1(\Real^3) $, then there exist $T=T(\|u_0\|_{\dot{H}^{s_c}_a\cap \dot{H}^1_a }) > 0$ and a unique solution $u$ to the IVP \eqref{PVI} satisfying
	$$
	u\in C\left([-T,T];\dot{H}^{s_c}_a\cap \dot{H}^1_a \right)\bigcap L^{q}\left([-T,T];\dot{H}_a^{s_c,p}\cap \dot{H}_a^{1,p} \right)\bigcap L^{\overline{q}}\left([-T,T]; L^{\overline{p}} \right),
	$$
	for any ($q,p$) $L^2$-admissible and $(\overline{q},\overline{p})$ $\dot H^{s_c}$-admissible. 
\end{thm}

We note that the nonlinear estimates used in the proof of the previous theorem are more refined then the previous ones that appeared in the literature (see, for instance, \cite[Lemmas 3.2-3.4]{CFG20}). Indeed, in Lemmas \ref{est0}, \ref{est2}, \ref{est1} we are able to achieve the power of $T$ predicted by the scaling of the equation. Such precise estimate implies the desired blow-up rate \eqref{BUA2} used as an additional  assumption in our previous works  \cite[Theorem 1.2.]{CF20}, \cite[Theorem 1.4]{CF23}. The strategy employed here is inspired by the new ideas introduced in \cite{CCF_Arxiv_22} and we split the spacial domain in a ball of radius $R$ and its complement and then optimize the choice of $R$.

%

Next, we study the existence of blow-up solutions and the lower bound for the blow-up rate.

\begin{thm}\label{scteo1} Let $a>0$, $0<b<\frac{3}{2}$ and $\frac{2-b}{3}<\sigma<2-b$. Consider $u_0\in \dot{H_a^{s_c}}(\Real^3)\cap \dot H_a^1(\Real^3) $ such that $E(u_0)\leq 0$.Then the maximal time of existence $T^{\ast}>0$ of the corresponding solution $u(t)$ to \eqref{PVI} is finite in each one of the following cases
\begin{itemize}
\item[$(a)$] $u_0$ is radial;
\item[$(b)$] $\sigma<\min\left\{2-b,\frac{2}{3}\right\}$.
\end{itemize}
\end{thm}

It should be emphasized that in case $(b)$ the radial assumption is not necessary and instead we need an extra restriction for the parameter $\sigma$. However, if $b>4/3$, then ${2-b}<\frac{2}{3}$ and the result covers all the intercritical range for $\sigma$ (see \eqref{restrictions}).

Let $\sigma_c=\frac{6\sigma}{2-b}$. In view of the Sobolev embedding $\dot{H}^{s_c} \subset L^{\sigma_c}$ and \eqref{equiv0}, we have
\begin{equation}\label{equiv_sigma_c}
\|u(t)\|_{L^{\sigma_c}}\lesssim \|u(t)\|_{\dot{H}^{s_c}}\sim \|u(t)\|_{\dot{H}_a^{s_c}}
\end{equation}
and then $\dot{H}_a^{s_c} \subset L^{\sigma_c}$. The next result states that this scaling invariant Lebesgue norm blows up with a certain lower bound.

\begin{thm}\label{scteo2}
Assume $a>0$, $0<b<\frac{3}{2}$ and $\frac{2-b}{3}<\sigma<2-b$. Let $u_0\in \dot{H_a^{s_c}}(\Real^3)\cap \dot H_a^1(\Real^3)$ be such that the maximal time of existence $T^{\ast}>0$ of the corresponding solution $u$ to \eqref{PVI} is finite.
Suppose that either $u_0$ is radial or $\sigma<\min\left\{2-b,\frac{2}{3}\right\}$. Then there exists $\alpha=\alpha(N,\sigma,b)>0$ such that
	\begin{align}\label{rate}
	  \|u(t)\|_{L^{\sigma_c}}\geq |\log (T^{\ast}-t)|^{\alpha},\,\,\,as\,\,\,t\to T^{\ast}.
	\end{align}
\end{thm} 


Note that this phenomenon does not occur in the mass-critical case ($s_c = 0$), since the scaling invariant $L^2$-norm is preserved along the flow by the mass conservation \eqref{Mass}. 

We emphasize that our previous studies about the INLS model in \cite{CF20, CF23} prove the lower bound \eqref{rate} upon imposing an additional assumption and only implies $\limsup_{t\to T^*} \|u(t)\|_{\dot{H}^{s_c}}=\infty$. The novelty here is that we prove that this assumption always holds in the intercritical case, see Proposition \ref{Blowalt} below. Using similar ideas, the same statement can also be shown to hold for the INLS setting, improving our previous results in this case.

We expect that our blow-up results in the last two theorems extend to higher dimensions, if one can generalize the local theory in Theorem \ref{LWP} to dimension $N\geq 3$. The main obstacle is the need of Strichartz estimates for exotic pairs, especially the non $L^2$-admissible ones. This can be achieved, for example, by proving the dispersive $L^1$-$L^\infty$ decay for the linear evolution, which can then be combined with the theory developed by Foschi \cite{Foschi05} to prove the desired estimates.

This paper is organized as follows. In Section \ref{Prel}, we introduce some notation and preliminary estimates. The well-posedness theory in $\dot{H_a^{s_c}}(\Real^3)\cap \dot H_a^1(\Real^3)$ is presented in Section \ref{WPT}. The last section is devoted to prove the existence of blow-up solutions and also a lower bound for the blow-up rate.

\section{Preliminaries}\label{Prel}

We start this section by introducing the notation used throughout the paper. We use $c$ to denote various constants that may vary line by line. Given any positive quantities $a$ and $b$, the notation $a\lesssim b$ means that $a \leq cb$, with $c$ uniform with respect to the set where $a$ and $b$ vary. We denote by $p'$ the H\"older conjugate of $1 \leq p \leq \infty$ and we
use $a^+$ and $a^-$ to denote $a +\varepsilon$ and $a -\varepsilon$, respectively, for a sufficiently small $\varepsilon > 0$.

The following  fractional Leibniz and Chain rules obtained will be used in the sequel (see also the generalizations suitable for singular weights proved in \cite{CCF_Arxiv_22}).
\begin{lemma}[Fractional calculus \cite{CW91, KPV93}]\label{leibniz} 
Let $N\geq 1$, $s \in (0,1]$, $p, q_1,q_2, r_1, r_2\in (1,\infty)$ with $\frac{1}{p} = \frac{1}{q_i} + \frac{1}{r_i}$, $i = 1, 2$ and $G\in C^1(\mathbb{C})$. Then
\begin{itemize}
\item[(i)] $\displaystyle \|D^s(fg)\|_{L^p(\mathbb R^N)}  \lesssim \|D^sf\|_{L^{q_1}(\mathbb R^N)}\|g\|_{L^{r_1}(\mathbb R^N)}+\|f\|_{L^{p_2}(\mathbb R^N)}\|D^s g\|_{L^{q_2}(\mathbb R^N)}.$\\
\item[(ii)] $\displaystyle \|D^sG(f)\|_{L^p(\mathbb R^N)}  \lesssim \|G'(f)\|_{L^{q_1}(\mathbb R^N)}\|D^sf\|_{L^{r_1}(\mathbb R^N)}.$
\end{itemize}

\end{lemma}

Next, we recall an important tool proved by Killip, Miao, Visan, Zhang and Zheng. It was proved for a large range of parameters, and we only state the case that interests us here. 
\begin{lemma}[Equivalence of Sobolev norms \cite{KMVZZ18}]\label{EquivHs} If $a >0$, $0<s<1$, and $1<p< 3/s$ then
	\begin{equation}
		\|D^s f\|_{L^p}\lesssim \| \mathcal{L}_a^{\frac{s}{2}} f\|_{L^p} \lesssim \|D^s f\|_{L^p}\;\textnormal{for all}\;f\in \mathbb{C}^\infty_c(\mathbb{R}^3 ).    
	\end{equation}
\end{lemma}

Now, we define the notion of admissible pairs.

\begin{definition}\label{Hs_adm} If $s \in (-1,1)$, the pair $(q,r)$ is called $\dot{H}^s$\textit{-admissible} if it satisfies the condition
\begin{equation}\label{hs_adm_eq}
\frac{2}{q} = \frac{3}{2}-\frac{3}{r}-s,    
\end{equation}
where
$$
2 \leq q,r \leq \infty.
$$
In particular, if $s=0$, we say that the pair is $L^2$-admissible.
\end{definition}

\begin{definition}\label{As}
Consider the set
\begin{equation}
    \mathcal{A}_0 = \left\{(q,r)\text{ is } L^2\text{-admissible} \left|
    \, 2 \leq r \leq 6
    \right.
    \right\}.
\end{equation}
For $s \in (-1,1)$, consider also
\begin{equation}
    \mathcal{A}_s = \left\{(q,r)\text{ is } \dot{H}^{s}\text{-admissible} \left|\,
    \left(\frac{6}{3-2|s|}\right)^+ \leq r \leq 6^-
    \right.
    \right\}.
\end{equation}
If $s \in [0,1)$, we define the following Strichartz norm
\begin{equation}
	\|u\|_{S(\dot{H}^s,I)} = \sup_{(q,r)\in \mathcal{A}_s}\|u\|_{L_I^qL_x^r},	
\end{equation}
and the dual Strichartz norm
\begin{equation}
	\|u\|_{S'(\dot{H}^{-s},I)} = \inf_{(q,r)\in \mathcal{A}_{-s}}\|u\|_{L_I^{q'}L_x^{r'}}.	
\end{equation}
For $s=0$, we shall write $S(\dot{H}^0,I) = S(L^2,I)$ and $S'(\dot{H}^0,I) = S'(L^2,I)$. If $I=\mathbb{R}$, we will omit $I$.
\end{definition}

Here, we are mainly interested in the case $a>0$ and $x\in \mathbb{R}^3$. In particular, the relation \eqref{equiv0} holds for $a>0$, $0<s<2$ and $r<3/s$, which implies 
\begin{align}\label{equivsc1}
	\|f\|_{\dot{H}^{s_c}_a}\sim \|f\|_{\dot{H}^{s_c}} \,\,\, \mbox{and} \,\,\, \|f\|_{\dot{H}^{1}_a}\sim \|f\|_{\dot{H}^{1}}.
\end{align}

%

In this work, we consider the following Strichartz estimates for $s \in [0,1)$ (see Kato \cite{Kato94}, Cazenave \cite{cazenave}, Keel and Tao \cite{KT98}, Foschi \cite{Foschi05})
\begin{equation}\label{S2}
\|e^{it\Delta}f\|_{S(\dot{H}^s)} \lesssim\|f\|_{\dot{H}^s},
\end{equation}
and
\begin{equation}\label{KS2}
\left\|\int_\mathbb{R}e^{i(t-\tau)\Delta}g(\cdot,\tau) \, d\tau\right\|_{S(\dot{H}^s,I)} + \left\|\int_0^te^{i(t-\tau)\Delta}g(\cdot,\tau) \, d\tau\right\|_{S(\dot{H}^s,I)}\lesssim\|g\|_{S'(\dot{H}^{-s},I)}.
\end{equation}
\begin{lemma}\label{Lemma-Str1}
	Let $s \in [0,1)$ and $a>-\frac{(N-2)^2}{4}$. Then,
	\begin{eqnarray}
		\| e^{-it \mathcal{L}_a}f \|_{S(\dot{H}^s;I)} &\lesssim &  \|f\|_{\dot{H}^s_a} \label{SE2}\\
		\left \| \int_{t_0}^t e^{-i(t-t')\mathcal{L}_a}g(\cdot,t') dt' \right \|_{S(L^2;I) } &\lesssim& \|g\|_{S'(L^2;I)} \label{SE3}
	\end{eqnarray} 
	\begin{proof}
		See Burq, Planchon, Stalker and Tahvildar-Zadeh \cite{BPST03} and Zhang and Zheng \cite{ZZ14}.
	\end{proof}
\end{lemma}

\begin{lemma}\label{Lemma-Str2}
	Let $s \in [0,1)$ and $a> 0$. Then,
	\begin{equation}\label{SE5}
		\left \| \int_{t_0}^t e^{-i(t-t')\mathcal{L}_a}g(\cdot,t') dt' \right \|_{S(\dot{H}^s;I) } \lesssim \|g\|_{S'(\dot{H}^{-s};I)},
	\end{equation}
	where $I\subset\mathbb{R}$ be an interval and $t_0\in I$.
	\begin{proof}
	The Application 2 in Fanelli, Felli, Fontelos and Primo \cite{FFFP13} proves

  \begin{equation}\label{dispersive_decay}
      \|e^{-it \mathcal{L}_a} f\|_{L^\infty_x} \lesssim \frac{1}{|t|^{\frac{3}{2}}}\|f\|_{L^1_x},
  \end{equation}
  which corresponds to the same decay as for the free ($a=0$) Schr\"odinger evolution. The estimate \eqref{SE5} is then a corollary of the main result in Foschi \cite{Foschi05}.
	\end{proof}
\end{lemma}

\begin{rem}
The result in Lemma \ref{Lemma-Str1} also holds in general space dimension $N>3$. On the other hand, Lemma \eqref{Lemma-Str2} is still an open problem in this general setting.
\end{rem}


\section{Well-posedness in {$\dot{H_a^{s_c}}\cap \dot H_a^1$}}\label{WPT}
In this section we study the well-posedness to  $(\mbox{INLS})_a$ in homogeneous Sobolev spaces $\dot{H_a^{s_c}}(\Real^3)\cap \dot H_a^1(\Real^3)$. The main ingredients are the equivalence of Sobolev norms (Lemma \ref{EquivHs}) and Strichartz estimates (Lemmas \ref{Lemma-Str1}-\ref{Lemma-Str2}).

\subsection{Nonlinear estimates}
We start by proving some preliminary estimates. The novelty here is that we can optimize the estimates around and outside the origin in order to establish the exact power of $T$ in the right hand side of these inequalities. This will be important to deduce a lower bound for the blow-up rate for finite time solutions (see Proposition \ref{Blowalt}).

\begin{lemma}\label{est0} Let $a>0$, $0<b<\frac{3}{2}$ and $\frac{2-b}{3}<\sigma<2-b$. Then, for $I=[-T,T]$ we have
\begin{equation}
\left\|\mathcal L_a^{\frac12}\left(|x|^{-b}|u|^{2\sigma}u\right) \right\|_{S'(L^2;I)} \lesssim T^{\sigma(1-s_c)}\|\mathcal L_a^\frac12 u\|^{2\sigma+1}_{S(L^2;I)}.
\end{equation}
\end{lemma}
\begin{proof}
Let $R>0$ to be chosen later and denote by $B_R$ the closed ball of radius $R$ in $\mathbb{R}^3$. We consider the sets $B^R_+=B_R$ and $B^R_-=\Real^3\backslash B_R$. Let $\eta\in(0,1)$ be such that $(1-2b)/3 < \eta < (3-2b)/2$ (which is possible since $0 < b < 3/2$). Let also $0<\varepsilon\ll\min\{(3-2b)/2-\eta, \eta - (1-2b)/3\}$. Define
\begin{equation}\label{palpha}
p = 6/(1+2 \eta) \quad \mbox{and} \quad \alpha_{\pm} = \frac{6(2\sigma+1)}{5-2b+4\sigma - 2\eta \mp 2\varepsilon}.
\end{equation}

From the choices of $\eta$ and $\varepsilon$, together with $(2-b)/3<\sigma<2-b$ and $0 < b < 3/2$, one gets that $2< p <6 $ and $2<\alpha_{\pm} <3$. Moreover, the following relation holds
\begin{eqnarray}\label{rel1}
\frac{1}{p'}&=&\frac{b+1\pm\varepsilon}{3}+(2\sigma+1)\left(\frac{1}{\alpha_{\pm}}-\frac{1}{3}\right)\\
&=&\frac{b\pm\varepsilon}{3}+2\sigma\left(\frac{1}{\alpha_{\pm}}-\frac{1}{3}\right)+\frac{1}{\alpha_{\pm}}.
\end{eqnarray}
Therefore, from the H\"older inequality and Sobolev embedding we obtain
\begin{align}
\left\|\nabla\left(|x|^{-b}|u|^{2\sigma}u\right)\right\|_{L^{p'}_x(B^R_{\pm})}
	&\lesssim \left(\left\||x|^{-b}\right\|_{L^{\frac{3}{b\pm \varepsilon}}_x(B^R_{\pm})}+\left\||x|^{-b-1}\right\|_{L^{\frac{3}{b+1\pm \varepsilon}}_x(B^R_{\pm})}\right)\|\nabla u\|_{L^{\alpha_{\pm}}_x}^{2\sigma+1}\\ 
	&\lesssim R^{\pm \varepsilon}\|\nabla u\|_{L^{\alpha_{\pm}}_x}^{2\sigma+1}.\label{estle1ger}
\end{align}

Now, let $q$, $\beta_{\pm}$ be such that $(q,p)$, $(\beta_+, \alpha_+)$ and $(\beta_-, \alpha_-)$ are $L^2$-admissible pairs. One can easily check that, for $\theta_\pm=\sigma(1-s_c)\mp \varepsilon/2$, the following condition is satisfied
\begin{eqnarray}\label{rel2}
\frac{1}{q'}=\theta_\pm+\frac{2\sigma+1}{\alpha_\pm},
\end{eqnarray}
then, by the H\"older inequality in the time variable and the equivalence of Sobolev norms, we have
\begin{align}
	\left\|\mathcal L_a^{\frac12}\left(|x|^{-b}|u|^{2\sigma}u\right)\right\|_{L^{q'}_TL^{p'}_x}
	&\lesssim  \left(T^{\theta_{+}}R^{\varepsilon}+T^{\theta_{-}}R^{-\varepsilon}\right)\|\mathcal L^\frac12_a u\|_{S(L^2;I)}^{2\sigma+1}.
\end{align}
Finally, taking $R=T^{\frac12}$, we have the desired estimate.
\end{proof}

\begin{lemma}\label{est2} Let $a>0$, $0<b<\frac32$ and $\frac{2-b}{3}<\sigma<2-b$. Then, for $I=[-T,T]$ we have
	\begin{equation}
		\left\|\mathcal L_a^{\frac{s_c}{2}}\left(|x|^{-b}|u|^{2\sigma}u\right) \right\|_{S'(L^2;I)} \lesssim T^{\sigma(1-s_c)}\|\mathcal L_a^\frac12 u\|_{S(L^2;I)}^{2\sigma}\|\mathcal L_a^\frac{s_c}{2} u\|_{S(L^2;I)}.
	\end{equation}
\end{lemma}

\begin{rem}
Here, due to the product rule, we are required to estimate quantities comparable to $(D^{s_c}|x|^{-b}) |u|^{2\sigma}u$ and to $|x|^{-b} |u|^{2\sigma}D^{s_c}u$. However, unlike the previous Lemma, since fractional derivatives are not localized in space, one cannot simply split the estimate on a ball and its exterior. Indeed, it is \textbf{not} expected that
\begin{equation}
    \|D^{s_c}(|x|^{-b}f)\|_{L_x^p(B_\pm^R)} \lesssim \|D^{s_c}|x|^{-b}\|_{L_x^{p_1}(B_\pm^R)}  \|f\|_{L^{q_1}} + \||x|^{-b}\|_{L_x^{p_2}(B_\pm^R)} \|D^s f\|_{L^{q_2}}
\end{equation}
for general $f$, even for suitable $p,p_j,q_j$, due to the non-local character of the fractional derivatives. A more appropriate version of the Leibniz rule, adapted to deal with this case, can be found in \cite{CCF_Arxiv_22}, and the nonlinear estimates there can be readily adapted to the proof of this Lemma. Nevertheless, we present another approach here, involving a well-applied localization, and a Gagliardo-Nirenberg inequality. Indeed, we let $\psi$ be a smooth, nonincreasing, radially symmetric function such that

\begin{equation}
    \psi(x) = \begin{cases}
        1, & |x| \leq 1\\
        0, & |x| \geq 2,
    \end{cases}
\end{equation}
and let $\psi_R (x) = \psi (x/R)$ for all $x$ and $R>0$. We then claim
\end{rem}
\begin{claim}\label{claim_x_b}
We have, for all $0\leq s \leq 1$,
\begin{equation}
    \|D^{s}(|x|^{-b}\psi_R)\|_{L^{\frac{3}{b+s+\varepsilon}}} \lesssim R^{\varepsilon},
\end{equation}
and
\begin{equation}
    \|D^{s}(|x|^{-b}(1-\psi_R))\|_{L^{\frac{3}{b+s-\varepsilon}}} \lesssim R^{-\varepsilon}.
\end{equation}
\end{claim}
\begin{proof}[Proof of the Claim]

We use Sobolev embedding to deduce
\begin{equation}
    \|D^{s}(|x|^{-b}\psi_R)\|_{L^{\frac{3}{b+s+\varepsilon}}}\lesssim  \|\nabla(|x|^{-b}\psi_R)\|_{L^{\frac{3}{b+1+\varepsilon}}} \lesssim R^{\varepsilon},
\end{equation}
since $|\nabla |x|^{-b}||\psi_R| = |x|^{-b-1}|\psi_R|$ and since $|x|^{-b}|\nabla \psi_R| \lesssim R^{-b-1}$ and $\nabla \psi_R$ is supported on $R \leq |x| \leq 2R$. Similar calculations give the estimate for the term with $(1-\psi_R)$.
\end{proof}

\begin{proof}[Proof of Lemma \ref{est2}]
%

As in the proof of Lemma \ref{est0}, consider the $L^2$-admissible pairs $(q,p)$, $(\beta_+, \alpha_+)$ and $(\beta_-, \alpha_-)$ given by \eqref{palpha}. From relation \eqref{rel1}, we also have
\begin{eqnarray}
\frac{1}{p'}&=&\frac{b+s_c\pm\varepsilon}{3}+2\sigma\left(\frac{1}{\alpha_{\pm}}-\frac{1}{3}\right)+\left(\frac{1}{\alpha_{\pm}}-\frac{s_c}{3}\right).
\end{eqnarray}
The Lemma \ref{leibniz} (Leibniz and Chain rules), H\"older and Sobolev inequalities yield
\begin{align}
    \|D^{s_c}(|x|^{-b}|u|^{2\sigma}u)\|_{L_x^{p'}}
    &\lesssim \|D^{s_c}(|x|^{-b}\psi_R)\|_{L_x^{\frac{3}{b+s_c+\varepsilon}}} \|\nabla u\|^{2\sigma}_{L_x^{\alpha_+}}\|D^{s_c}u\|_{L_x^{\alpha_+}} \\
    &\quad+ \||x|^{-b}\psi_R\|_{L_x^{\frac{3}{b+\varepsilon}}} \|\nabla u\|^{2\sigma}_{L_x^{\alpha_+}}\|D^{s_c}u\|_{L_x^{\alpha_+}}\\
    &\quad +\|D^{s_c}(|x|^{-b}(1-\psi_R))\|_{L_x^{\frac{3}{b+s_c-\varepsilon}}} \|\nabla u\|^{2\sigma}_{L_x^{\alpha_-}}\|D^{s_c}u\|_{L_x^{\alpha_-}}\\
    &\quad+ \||x|^{-b}(1-\psi_R)\|_{L_x^{\frac{3}{b-\varepsilon}}} \|\nabla u\|^{2\sigma}_{L_x^{\alpha_-}}\|D^{s_c}u\|_{L_x^{\alpha_-}}.
\end{align}
Next, from Claim \ref{claim_x_b}, we get
$$
\|D^{s_c}(|x|^{-b}|u|^{2\sigma}u)\|_{L_x^{p'}} \lesssim R^{\varepsilon}\|\nabla u\|^{2\sigma}_{L_x^{\alpha_+}}\|D^{s_c}u\|_{L_x^{\alpha_+}} +
    R^{-\varepsilon}\|\nabla u\|^{2\sigma}_{L_x^{\alpha_-}}\|D^{s_c}u\|_{L_x^{\alpha_-}}.
$$

Finally, using the relation \eqref{rel2}, the H\"older inequality in the time variable and the equivalence of Sobolev norms imply
\begin{equation}
\|\mathcal L^\frac{s_c}{2}_a(|x|^{-b}|u|^{2\sigma}u)\|_{L^{q'}_TL^{p'}_x} \lesssim \left(T^{\theta_{+}}R^{\varepsilon}+T^{\theta_{-}}R^{-\varepsilon}\right)\|\mathcal L^\frac12_a u\|_{S(L^2;I)}^{2\sigma}\|\mathcal L^\frac{s_c}{2}_a u\|_{S(L^2;I)}
\end{equation}
and the choice $R = T^{\frac12}$ finishes the proof.
\end{proof}

\begin{lemma}\label{est1} Let $a>0$, $0<b<2$, $\frac{2-b}{3}<\sigma<2-b$. Then, there exists $\delta \in (0,1)$ such that, for $I=[-T,T]$ we have
	\begin{equation}
		\left\|\mathcal |x|^{-b}|u|^{2\sigma}v \right\|_{S'(\dot H^{-s_c};I)} \lesssim \left[\left(T^{\frac{1-s_c}{2}}\|\mathcal L_a^\frac{1}{2} u\|_{S(L^2;I)}\right)^{1-\delta}\left(\|\mathcal L_a^\frac{s_c}{2} u\|_{S(L^2;I)}\right)^\delta\right]^{2\sigma}\|v\|_{S(\dot H^{s_c}; I)}.
	\end{equation}

\end{lemma}
\begin{proof}
For a fixed $\sigma>0$, take $s\in (s_c,1)$ such that
\begin{equation}\label{sigma_cond}
\sigma<\frac{1-s_c}{s-s_c}.
\end{equation}
So, since $2\sigma(s-s_c)+1<3-2s_c$, we can find $\ell$ satisfying
$$\max\left\{\frac{6}{3-2s_c},\frac{6}{2\sigma(s-s_c)+1+\sigma-\varepsilon}\right\}< \ell<\frac{6}{2\sigma(s-s_c)+1+\varepsilon},
$$	
for some $0<\varepsilon\ll 1$, such that $(m,\ell)$ is $\dot H^{-s_c}$-admissible and $(\tau,\ell)$ is $\dot H^{s_c}$-admissible. Also consider the pairs $(\nu_\pm,\mu_\pm)$ such that
\begin{align}
	\mu_{\pm}=\frac{6\sigma \ell}{\ell(3+2\sigma s-b\mp\varepsilon)-6}.
\end{align}
It is straightforward to check that $2<\mu_\pm<3$, so we can define the $L^2$-admissible pairs $(\nu_-,\mu_-)$ and $(\nu_+,\mu_+)$. Now, setting 
$$
\zeta_\pm=\sigma(s-s_c)\mp\frac{\varepsilon}{2},
$$ 
we get
\begin{align}
	\frac{1}{\ell'}=\frac{b\pm\varepsilon}{3}+2\sigma\left(\frac{1}{\mu_\pm}-\frac{s}{3}\right)+\frac{1}{\ell}
\,\,\,\,
\mbox{and}
\,\,\,\,
		\frac{1}{m'}=\zeta_\pm+\frac{2\sigma}{\nu_\pm}+\frac1\tau.
\end{align}
Then, applying H\"older's inequality and Sobolev embedding we have, using the same notation of the previous lemmas, that
\begin{align}
	\left\||x|^{-b}|u|^{2\sigma}v\right\|_{L^{m'}_TL^{\ell'}_x(B^R_{\pm})}
	&\lesssim T^{\zeta_{\pm}} R^{\pm \varepsilon}\|D^s u\|_{L^{\nu_{\pm}}_TL^{\mu_{\pm}}_x}^{2\sigma}\|v\|_{L^{\tau}_TL^{\ell}_x}.
\end{align}
Moreover, by interpolation and the equivalence of Sobolev norms \eqref{equiv0} we obtain
\begin{align}
\|D^s u\|_{L^{\nu_{\pm}}_TL^{\mu_{\pm}}_x}&\lesssim \|D u\|^{1-\delta}_{L^{\nu_{\pm}}_TL^{\mu_{\pm}}_x}\|D^{s_c} u\|^{\delta}_{L^{\nu_{\pm}}_TL^{\mu_{\pm}}_x}\\
&\lesssim \|\mathcal L^{\frac{1}{2}}_a u\|_{L^{\nu_{\pm}}_TL^{\mu_{\pm}}_x}^{1-
\delta}\|\mathcal L^{\frac{s_c}{2}}_a u\|_{L^{\nu_{\pm}}_TL^{\mu_{\pm}}_x}^{\delta},
\end{align}
for $\delta = (1-s)/(1-s_c)\in (0,1)$. Collecting the previous two estimates, we obtain
\begin{align}
\left\||x|^{-b}|u|^{2\sigma}v\right\|_{L^{m'}_TL^{\ell'}_x}
	&\lesssim  \left(T^{\zeta_{+}}R^{\varepsilon}+T^{\zeta_{-}}R^{-\varepsilon}\right)\left(\|\mathcal L_a^\frac{1}{2} u\|_{S(L^2;I)}^{1-\delta}\|\mathcal L_a^\frac{s_c}{2} u\|_{S(L^2;I)}^{\delta}\right)^{2\sigma}\|v\|_{S(\dot H^{s_c}; I)}.
\end{align}

Again setting $R=T^{\frac12}$, we find
\begin{equation}\label{T_cond}
T^{\zeta_{+}}R^{\varepsilon}+T^{\zeta_{-}}R^{-\varepsilon}=T^{\sigma(s-s_c)} = T^{\sigma(1-\delta)(1-s_c)},
\end{equation}
completing the proof.

\end{proof}


\subsection{Local theory}
Next, we use the above nonlinear estimates to obtain the local well-posedness in $ \dot{H_a^{s_c}}(\Real^3)\cap \dot H_a^1(\Real^3) $.

\begin{proof}[Proof of Theorem \ref{LWP}]	
	Let
	$$
	X=  \left( \bigcap_{(q,p)\in \mathcal{A}_0}L^q\left([-T,T];\dot{H}^{s_c,p}_a\cap \dot{H}^{1,p}_a\right)\right)\bigcap \left( \bigcap_{(\overline{q},\overline{p})\in \mathcal{A}_{s_c}}L^{\overline{q}}\left([-T,T]; L^{\overline p} \right)\right)$$ and 
	\begin{equation*}\label{NHs} 
		\|u\|_T=\|\mathcal L_a^{\frac12} u\|_{S\left(L^2;I\right)}+\|\mathcal L_a^{\frac{s_c}{2}} u\|_{S\left(L^2;I\right)}+\|u\|_{S\left(\dot H^{s_c};I\right)},
	\end{equation*}
	where $I=[-T,T]$.
	
	For {$m>1$} and $T>0$ define the set
	\begin{equation*}
		S(m,T)=\{u \in X : \|u\|_T\leq m \}
	\end{equation*}
	with the metric 
	$$
	d_T(u,v)=\|u-v\|_{S\left(\dot H^{s_c};I\right)}.
	$$
	
	By standards arguments $(S(m,T), d_T)$ is a complete metric space (see Cazenave \cite[Theorem 4.4.1]{cazenave} and also \cite[Appendix A]{CFG20}). We shall show that $G=G_{u_0}$ defined by 
	\begin{equation}\label{OPERATOR} 
	G(u)(t)=e^{-it\mathcal L_a}u_0+i\lambda \int_0^t e^{-i(t-s)\mathcal L_a}\left(|x|^{-b}|u(s)|^{2\sigma} u(s)\right)ds,
	\end{equation}
has a fixed point in $S(m,T)$ for a suitable choice of $m$ and $T$. Indeed, it follows from the Strichartz inequalities in Lemma \ref{Lemma-Str1}, and also, from Lemmas \ref{est0} and \ref{est2}, that, for some $c>1$,
	\begin{align}\label{esa}
		\|\mathcal L_a^{\frac12}G(u)\|_{S(L^2;I)}
		&\leq c \|\mathcal L_a^\frac12 u_0\|_{L^2}+cT^{\sigma(1-s_c)} \|\mathcal L_a^\frac12 u\|^{2\sigma+1}_{S(L^2;I)},	
\end{align}
and
\begin{align}\label{esa0}
		\|\mathcal L_a^{\frac{s_c}{2}}G(u)\|_{S(L^2;I)}
		&\leq c \|\mathcal L_a^\frac{s_c}{2} u_0\|_{L^2}+cT^{\sigma(1-s_c)}\|\mathcal L_a^\frac12 u\|_{S(L^2;I)}^{2\sigma}\|\mathcal L_a^\frac{s_c}{2} u\|_{S(L^2;I)}.
	\end{align}
Moreover, from Lemma \ref{Lemma-Str2} and Lemma \ref{est1}, we also have
$$
\|G(u)\|_{S\left(\dot H^{s_c};I\right)}\leq c \|\mathcal L_a^\frac{s_c}{2} u_0\|_{L^2}+c\left[\left(T^{\frac{1-s_c}{2}}\|\mathcal L_a^\frac{1}{2} u\|_{S(L^2;I)}\right)^{1-\delta}\left(\|\mathcal L_a^\frac{s_c}{2} u\|_{S(L^2;I)}\right)^\delta\right]^{2\sigma}\|u\|_{S(\dot H^{s_c}; I)}.	
$$


Now, by fixing $m\geq \max\{1, 2c\|u_0\|_{\dot{H}^{s_c}_a\cap \dot{H}^1_a }\}$ and choosing $T\in (0,1)$ such that 
	\begin{equation}\label{CTHs} 
		T^{\sigma(1-s_c)} < \left(\frac{1}{4c m^{2\sigma}}\right)^{\frac{1}{1-\delta}},
	\end{equation}
we obtain that $G(u)\in S(m,T)$, if $u \in S(m,T)$ and therefore $G$ is well defined.
	To prove that $G$ is a contraction we first recall the elementary inequality
	\begin{equation}\label{nonlinearity}
		\left||x|^{-b}|u|^{2\sigma} u-|x|^{-b}|v|^{2\sigma} v\right|\lesssim |x|^{-b}\left( |u|^{2\sigma}+ |v|^{2\sigma} \right)|u-v|.
	\end{equation}
	Then, again from Lemma \ref{Lemma-Str2} and Lemma \ref{est1}, we get for $u,v\in S(m,T)$ that

	\begin{align*}
d_T(G(u),G(v)) \leq 2c T^{\sigma(1-s_c)(1-\delta)}m^{2\sigma} d_T(u,v).
	\end{align*}
	
Therefore, for $T\in (0,1)$ satisfying \eqref{CTHs}, $G$ is a contraction on $S(m,T)$. Finally, by the contraction mapping principle we have a unique $u \in S(m,T)$ such that $G(u)=u$ and the proof is completed.

\end{proof}

\subsection{A Gagliardo-Nirenberg type inequality and Blow-up alternative}
We start recalling the celebrated Caffareli-Khon-Nirenberg inequality in $\mathbb{R}^3$.
\begin{prop}[Caffareli-Khon-Nirenberg inequality]\label{GNcrit2} There exists $C>0$ such that the following inequality holds for all $f \in \dot H^1(\mathbb{R}^3)\cap L^p(\Real^3)$
	\begin{equation}\label{GNp}
		\int |x|^{-b}|f|^{2\sigma+2} \, dx\leq C \|\nabla f\|^\theta_{L^2}\|f\|^{2\sigma+2-\theta}_{L^{p}}.
	\end{equation}
	if and only, if the following relations hold
	\begin{align}
		\frac{3-b}{3}=\frac{\theta}{6}+\frac{2\sigma+2-\theta}{p}
	\end{align} 
	where $b<3$, $p\geq 1$ and $0<\theta<2\sigma+2$.
\end{prop}
\begin{proof}
	This is the 3D version of the main result in Caffareli, Khon and Nirenberg \cite{CKN84}.
\end{proof}
\begin{rem}We note that:

	\begin{itemize}
		\item[1.] Taking $p=2$ in the inequality \eqref{GNp}, we have the well-known Gagliardo-Nirenberg inequality
		\begin{align}\label{GNF}
			\int|x|^{-b}|f|^{2\sigma+2}\,dx\leq C\|\nabla f\|_{L^2}^{2\sigma s_c+2}\|f\|_{L^2}^{2\sigma(1-s_c)},
		\end{align}
		where the best constant was established by the third author in \cite{Farah}.
		\item[2.] For $p=\sigma_c=\frac{2N\sigma}{2-b}$, the inequality \eqref{GNp} is given by	\begin{equation}\label{GNcrit}
			\int |x|^{-b}|f|^{2\sigma+2} \, dx\leq C \|\nabla f\|^2_{L^2}\|f\|^{2\sigma}_{L^{\sigma_c}}.
		\end{equation}
		In \cite{CFG20}, joint with Guzm\'an, the sharp constant for above inequality was obtained.
		\item[3.] Since $a>0$, we have that (see relation \eqref{equiv})
		$$\|\mathcal L_a^{\frac12} u\|_{L^2}\sim \|\nabla u\|_{L^2}.$$
		Thus, it follows from inequality \eqref{GNp} that there exists an optimal constant $K_{a}>0$ such that 
		\begin{equation}\label{GNacrit2}
			\int |x|^{-b}|f|^{2\sigma+2} \, dx\leq \frac{\sigma+1}{K_a^{2\sigma}} \|\mathcal L_a^{\frac12}
 f\|^2_{L^2}\|f\|^{2\sigma}_{L^{\sigma_c}}.
		\end{equation}
	\end{itemize}
We will refer to the constant $K_{a}$ later in this work (see Proposition \ref{prop1}).
\end{rem}
Next, we prove a blow-up alternative which follows as a consequence of the local theory obtained in Theorem \ref{LWP}.
\begin{prop}[Blow-up alternative]\label{Blowalt}
	Let $a>0$, $0<b<\frac32$ and $\frac{2-b}{3}<\sigma<2-b$. If $u$ is a solution to the IVP \eqref{PVI} in $\dot{H_a^{s_c}}(\Real^3)\cap \dot H_a^1(\Real^3)$ with finite maximal positive time of existence $0<T^*<+\infty$, then there exists a constant $c>0$ such that
\begin{align}\label{BUA2}
		\|\mathcal L_a^\frac12 u(t)\|_{L^2}\geq \frac{c}{(T^*-t)^{\frac{1-s_c}{2}}}, \,\,\, \textit{for all} \,\,\, t\in [0,T^*).
\end{align}
\end{prop}
\begin{proof} 


Assume, 
by contradiction, that there exists a sequence $t_n\nearrow T^{\ast}$ such that, for all $n$,
\begin{equation}\label{bound_0}
    (T^*-t_n)^{\frac{1-s_c}{2}}\|\mathcal L_a^\frac12 u(t_n)\|_{L^2} < \frac{1}{n}.
\end{equation}

Using the same notation as in the proof of Lemma \ref{est0}, we define
\begin{equation}
    X_n(t) =\|\mathcal L_a^{\frac12}u\|_{L^{\beta_+}_{[t_n,t]}L^{\alpha_+}_x} + \|\mathcal L_a^{\frac12}u\|_{L^{\beta_-}_{[t_n,t]}L^{\alpha_-}_x}.
\end{equation}
From the Duhamel formula and estimate \eqref{bound_0}, for $t \in [t_n, T^*)$, there exists $C_0>0$ such that
\begin{equation}\label{ContArg}
    \|\mathcal L_a^{\frac12}u\|_{L^{\infty}_{[t_n,t]}L^{2}_x} + 
   X_n(t)\leq  
    C_0\|\mathcal L_a^\frac12 u(t_n)\|_{L^2}  + \frac{C_0}{(n \|\mathcal L_a^\frac12 u(t_n)\|_{L^2})^{2\sigma}} [X_n(t)]^{2\sigma+1}.
\end{equation}

Consider the function $f(x)=x-C_0\|\mathcal L_a^\frac12 u(t_n)\|_{L^2} - \frac{C_0}{(n \|\mathcal L_a^\frac12 u(t_n)\|_{L^2})^{2\sigma}}  x^{2\sigma+1}$. A simple computation revels that it has a global maximum at $x_n^{\ast}=\frac{n\|\mathcal L_a^\frac12 u(t_n)\|_{L^2}}{[C_0(2\sigma+1)]^{1/2\sigma}}$, $f(x_n^{\ast})=\left(\frac{2\sigma n}{(2\sigma+1)[C_0(2\sigma+1)]^{1/2\sigma}}-C_0\right)\|\mathcal L_a^\frac12 u(t_n)\|_{L^2}$ and $f(0)<0$. Thus, for $n_0> \frac{\left[C_0(2\sigma+1)\right]^{\frac{2\sigma+1}{2\sigma}}}{2\sigma}$ we have that $f(x_{n_0}^{\ast})>0$.

Now, from \eqref{ContArg} we have that $f(X_{n_0}(t))\leq 0$, for all $t \in [t_{n_0}, T^*)$. A continuity argument then implies that $X_{n_0}(t)\leq x_{n_0}^{\ast}$ and therefore, from inequality \eqref{ContArg}, we deduce

\begin{eqnarray}\label{bound_1}
\|\mathcal L_a^{\frac12}u\|_{L^{\infty}_{[t_{n_0},T^*)}L^2_x} 
&\leq&  C_0\|\mathcal L_a^\frac12 u(t_{n_0})\|_{L^2}+ \frac{C_0}{(n_0 \|\mathcal L_a^\frac12 u(t_{n_0})\|_{L^2})^{2\sigma}} (x_{n_0}^*)^{2\sigma+1}\\
& =& \left(1+\frac{n_0}{[C_0(2\sigma+1)]^{(2\sigma+1)/2\sigma}} \right)C_0\|\mathcal L_a^\frac12 u(t_{n_0})\|_{L^2},
\end{eqnarray}
which implies 
$\|\mathcal{L}_a^{\frac{1}{2}}u\|_{L^{\infty}_{[0,T^*)}L^2_x} <\infty$, by continuity. 

On the other hand, the bound $X_{n_0}(t)\leq x_{n_0}^{\ast}$ and assumption \eqref{bound_0} also implies
\begin{equation}
    (t - t_{n_0})^{\frac{1-s_c}{2}}X_{n_0}(t)< \frac{1}{[C_0(2\sigma+1)]^{1/2\sigma}}, \,\,\mbox{for all}\,\, t\in [0,T^*).
\end{equation}
Moreover, again from Duhamel and the proof of estimate \eqref{esa0}, we can write, for $t_{n_0}\leq t < T^*$,

\begin{align}
   \|\mathcal L_a^{\frac{s_c}2}u\|_{S(L^2;[t_{n_0},t])} &\leq  
    C_0\|\mathcal L_a^\frac{s_c}2 u(t_{n_0})\|_{L^2}  + C_0 \left[(t-t_{n_0})^{\frac{1-s_c}{2}}X_{n_0}(t)\right]^{2\sigma}\|\mathcal L_a^{\frac{s_c}2}u\|_{S(L^2;[t_{n_0},t])}\\
    &\leq  C_0\|\mathcal L_a^\frac{s_c}2 u(t_{n_0})\|_{L^2}  + \frac{1}{2\sigma+1}\|\mathcal L_a^{\frac{s_c}2}u\|_{S(L^2;[t_{n_0},t])}
    .
\end{align}

We can then absorb the last term of the right-hand side in the left-hand side and let $t \nearrow T^*$ to obtain
\begin{equation}
     \|\mathcal L_a^{\frac{s_c}2}u\|_{S(L^2;[t_{n_0},T^{*}))} \leq  
    \frac{(2\sigma+1) C_0}{2\sigma}\|\mathcal L_a^\frac{s_c}2 u(t_{n_0})\|_{L^2},
\end{equation}
which is a contradiction, since $\|\mathcal{L}_a^{\frac{1}{2}}u\|_{L^{\infty}_{[0,T^*)}L^2_x} + \|\mathcal{L}_a^{\frac{s_c}{2}}u\|_{L^{\infty}_{[0,T^*)}L^2_x}$ cannot be finite, otherwise $T^{*} = +\infty$.

\begin{rem}
    The constant $c>0$ in the previous proposition can be chosen independently of $u$. Indeed, from its proof,one can choose
\begin{equation}
    c = \frac{2\sigma}{\left[C_0(2\sigma+1)\right]^{\frac{2\sigma+1}{2\sigma}}}.
\end{equation}
\end{rem}

\begin{rem}
    The previous proposition can be adapted to the INLS equation \eqref{INLS}, showing that the lower bound
$$
	\|\mathcal \nabla u(t)\|_{L^2}\geq \frac{c}{(T^*-t)^{\frac{1-s_c}{2}}}, \,\,\, \textit{for all} \,\,\, t\in [0,T^*),
$$
always hold for $\dot{H}^{s_c}\cap \dot{H}^1$ blow-up solution in the INLS setting, improving our previous results \cite[Theorem 1.2 and Corollary 1.3]{CF20} and  \cite[Theorem 1.4 and Corollary 1.5]{CF23}.
\end{rem}

%
%
\end{proof}

\section{Blow-up solutions in $\dot H_a^{s_c}\cap \dot H_a^1$}\label{sec4}

\subsection{A virial-type estimate}

Consider $u\in C([0,\tau_*]: \dot H_a^{s_c}\cap \dot H_a^1)$ a solution to \eqref{PVI} and define the function
\begin{align}\label{virial}
	z(t)=\displaystyle\int\phi|u(t)|^2\,dx,
\end{align} 
where $\phi$ is cut-off function in $\Real^3$.
In \cite{LUCCASCARLOS}, the second author and Guzm\'an obtained
\begin{equation}\label{zR'2}
	z'(t)=2\,\mbox{Im}\int \nabla\phi\cdot\nabla u(t)\overline{u}(t)\,dx
\end{equation}
and 
\begin{align}\label{zR''}
	z''(t)=&4\,\mbox{Re} \sum_{j,k=1}^{3}\int \partial_ju(t)\,\partial_k\overline u(t)\,\partial^2_{jk}\phi\,dx-\int |u(t)|^2 \Delta^2\phi\,dx-2a\int \nabla (|x|^{-2})\cdot \nabla \phi|u(t)|^{2}\,dx\nonumber\\
	&-\frac{2\sigma}{\sigma+1}\int|x|^{-b}|u(t)|^{2\sigma+2}\Delta\phi\,dx+\frac{2}{\sigma+1}\int\nabla\left(|x|^{-b}\right)\cdot \nabla\phi|u(t)|^{2\sigma+2}\,dx.
\end{align}

Set $\phi$ to be a non-negative radial function $\phi\in C^{\infty}_0(\Real^3)$, such that
\begin{align}\label{phi}
	\phi(x)=
	\left\{
	\begin{array}{ll}
		|x|^2,&\mbox{ se }|x|\leq 2\\
		0,&\mbox{ se }|x|\geq 4
	\end{array}
	\right.
\end{align}
satisfying
\begin{align}\label{nablaphi}
	\phi(x)\leq c|x|^2, \quad |\nabla \phi (x)|^2\leq c\phi(x) \quad \mbox{and} \quad \partial_r^2\phi(x)\leq  2, \quad \mbox{for all } |x|>2,
\end{align}
with $r=|x|$. Then, define $\phi_R(x)=R^2\phi\left(\frac{x}{R}\right)$. 
Since $\phi$ is radial, we have
\begin{align}\label{virial1}
	z'(t)=2\,\mbox{Im}\int \partial_r\phi_R\frac{x\cdot \nabla u(t)}{r}\overline{u}(t)\,dx
\end{align}
and 
\begin{align}\label{virial2}
	z_R ''(t)=&4\int \frac{\partial_r\phi_R}{r}|\nabla u(t)|^2\,dx+4\int \left(\frac{\partial_r^2\phi_R}{r^{2}}-\frac{\partial_r \phi_R}{r^3}\right)|x\cdot \nabla u(t)|^2\,dx+4a\int \frac{\partial_r \phi_R}{r}|x|^{-2}|u(t)|^2\,dx  \nonumber\\
	&-\int|u(t)|^2 \Delta^2\phi_R
	\,dx-\frac{2\sigma}{\sigma+1}\int \left[\partial^2_r\phi_R +\left(N-1+\frac{b}{\sigma}\right)\frac{\partial_r \phi_R}{r}\right]|x|^{-b}|u(t)|^{2\sigma+2}\,dx,
\end{align}
where $\partial_r$ denotes the derivative with respect to $r=|x|$.

In the following lemma we prove a virial type estimate adapted to the $(\mbox{INLS})_a$ setting.

\begin{lemma}\label{lemintradial}
	If $u\in C([0,\tau_*]: \dot{H_a^{s_c}}(\Real^3)\cap \dot H_a^1(\Real^3) )$ is a solution to \eqref{PVI} with initial data $u(0)=u_0$, then there exists $c>0$ depending only on $N, \sigma, b$ such that for all $R>0$ and $t\in[0, \tau_\ast]$ we have
	\begin{align}
	8\sigma s_c\|\mathcal L_a^{\frac12}
 u(t)\|^2_{L^2}&+z''_R(t)-16(\sigma s_c+1)E[u_0]\\
 \lesssim &\left(\frac{1}{R^2}\int_{2R\leq |x|\leq 4R}|u(t)|^2\,dx+a\int_{|x|\geq R} |x|^{-2}|u(t)|^2\,dx+\int_{|x|\geq R}|x|^{-b}|u(t)|^{2\sigma+2}\,dx\right).\label{estintradial}
	\end{align}
\end{lemma}

\begin{proof} Define
\begin{align}
	P[u]&=\int |\nabla u|^{2}\,dx+a\int |x|^{-2}|u|^2\,dx-\frac{3\sigma+b}{2\sigma+2}\int |x|^{-b}|u|^{2\sigma+2}\,dx,\label{P}
\end{align}
then can rewrite $z_R''(t)$ in \eqref{virial2} as
\begin{align}\label{z''K}
	z_R''(t)=8P[u(t)]+K_1+K_2+K_3+K_4,
\end{align}
where
\begin{align}
	K_1=&4\int \left(\frac{\partial_r\phi_R}{r}-2\right)|\nabla u(t)|^2\,dx+4\int \left(\frac{\partial^2_r\phi_R}{r^2}-\frac{\partial_r \phi_R}{r^3}\right)|x\cdot \nabla u(t)|^2\,dx,\label{K1}\\
	K_2=&4a\int \left(\frac{\partial_r \phi_R}{r}-2\right)|x|^{-2}|u(t)|^{2}\,dx\label{K2}\\
	K_3=&-\frac{2\sigma}{\sigma+1}\int\left[\partial^2_r\phi_R+\left(2+\frac{b}{\sigma}\right)\frac{\partial_r \phi_R}{r}-6-\frac{2b}{\sigma}\right]|x|^{-b}|u(t)|^{2\sigma+2}\,dx,\label{K3}\\
	K_4=&-\int|u(t)|^2\Delta^2\phi_R\,dx\label{K4}.
\end{align}

First note, by definition of $\phi_R$, that 
\begin{align}
	K_4\lesssim \frac{1}{R^{2}}\int_{2R\leq |x|\leq 4R} |u(t)|^{2}\,dx\label{K_4}.
\end{align}

Moreover, if $0\leq |x|\leq 2R$, then
\begin{align}
	\partial_r\phi_R(x)=2|x|=2r,\,\,\,\,\,\,\partial^2_r\phi_R(x)=2,
\end{align}
which implies
\begin{align}
	\mbox{supp} \left[\partial^2_r\phi_R+\left(2+\frac{b}{\sigma}\right)\frac{\partial_r\phi_R}{r}-6-\frac{2b}{\sigma}\right]\,\cup\,\, \mbox{supp} \left[\frac{\partial_r\phi_R}{r}-2\right]\subset (2R,\infty).
\end{align}
Therefore, there exists $c>0$ such that
\begin{align}
	K_3\lesssim& \int_{|x|\geq R}|x|^{-b}|u(t)|^{2\sigma+2}\,dx\label{K_3}
\end{align}
and
\begin{align}
K_2\lesssim a\int_{|x|\geq R} |x|^{-2}|u(t)|^2\,dx.\label{K_2}
\end{align}

Now, we set
\begin{align}
	\Omega=\left\{x\in \mathbb R^3;\,\,\frac{\partial^2_r\phi_R(x)}{r^2}-\frac{\partial_r\phi_R(x)}{r^3}\leq 0\right\}.
\end{align}
Since $\partial^2_r\phi_R\leq 2$ it follows that $\partial_r\phi_R(|x|)\leq 2|x|$, for all $x\in \mathbb R^3$. Then, using the Cauchy-Schwartz inequality, we get
\begin{align}
	K_1=&\,\,4\int\left(\frac{\partial_r\phi_R}{r}-2\right)|\nabla u(t)|^2\,dx+4\int\left(\frac{\partial_r^2\phi_R}{r^2}-\frac{\partial_r\phi_R}{r^3}\right) |x\cdot \nabla u(t)|^2\,dx\nonumber\\
	=&\,\,4\int_{\Omega} \left(\frac{\partial_r\phi_R}{r}-2\right)|\nabla u(t)|^2\,dx+4\int_{\Omega} \left(\frac{\partial_r^2\phi_R}{r^2}-\frac{\partial_r\phi_R}{r^3}\right) |x\cdot \nabla u(t)|^2\,dx\nonumber\\
	&+4\int_{\mathbb R^3\backslash\Omega} \left(\frac{\partial_r\phi_R}{r}-2\right)|\nabla u(t)|^2\,dx+4\int_{\mathbb R^3\backslash\Omega} \left(\frac{\partial_r^2\phi_R}{r^2}-\frac{\partial_r\phi_R}{r^3}\right) |x\cdot \nabla u(t)|^2\,dx\nonumber\\
	\leq&\,\, 4\int_{\mathbb R^3\backslash\Omega} \left(\frac{\partial_r\phi_R}{r}-2\right)|\nabla u(t)|^2\,dx+4\int_{\mathbb R^3\backslash\Omega} \left(\partial_r^2\phi_R-\frac{\partial_r\phi_R}{r}\right)\frac{|x|^2}{r^2}|\nabla u(t)|^2\,dx\nonumber\\
	=&\,\, 4\int_{\Real^3\backslash \Omega} \left(\partial_r^2\phi_R-2\right)|\nabla u(t)|^2\,dx\leq 0,\label{K_1}
\end{align}
where in the last inequality we have used the assumption \eqref{nablaphi}.

Collecting \eqref{K_4}, \eqref{K_3}, \eqref{K_2} and \eqref{K_4}, there exists $c>0$ such that
\begin{align}
	z_R''(t)-8P[u(t)]\lesssim \frac{1}{R^2}\int_{2R\leq |x|\leq 4R}|u(t)|^2\,dx+ a\int_{|x|\geq R} |x|^{-2}|u(t)|^2\,dx+\int_{|x|\geq R}|x|^{-b}|u(t)|^{2\sigma+2}\,dx\label{zR8Q}.
\end{align}

Finally, from the energy conservation \eqref{Energy}, we can write
\begin{align}
	P[u(t)]= -\sigma s_c\|\mathcal L_a^{\frac12}
 u(t)\|_{L^2}^2+2(\sigma s_c+1)E[u_0]
\end{align}
and combining the last two relations, we deduce the estimate \eqref{estintradial}.
\end{proof}
\subsection{Gagliardo-Nirenberg inequality}
 
We first recall the following scaling invariant Morrey-Campanato type semi-norm, used in \citet{MR_Bsc} 
\begin{equation}\label{defrho}
\rho(u,R)=\sup_{R'\geq R}\frac{1}{(R')^{2s_c}}\int_{R'\leq |x|\leq 2R'}|u|^2\,dx.
\end{equation}
It is easy to see that $\rho(u,R)$ is non-increasing in $R>0$. Moreover, by H\"older's inequality, there exists a universal constant $c>0$ such that for all $u\in L^{\sigma_c}$ and $R>0$
\begin{equation}\label{radial1}
	\frac{1}{R^{2s_c}}\int_{|x|\leq R}|u|^2\,dx\leq c\|u\|_{L^{\sigma_c}}^2
\end{equation}
and
\begin{equation}\label{radial2}
	\lim_{R\to+\infty}\frac{1}{R^{2s_c}}\int_{|x|\leq R}|u|^2\,dx=0
\end{equation}
(c.f. Merle and Raph\"ael in \citep[Lemma 1]{MR_Bsc} and also \cite[Lemma 2.1]{CF20}). In \cite{CF20} and \cite{CF23}, the last two authors proved some interpolation inequalities adapted to the inhomogeneous case in $\dot{H}^{s_c}\cap\dot{H}^{1}$. The relations \eqref{equivsc1} imply that these results also hold in the space $\dot{H}_a^{s_c}\cap\dot{H}_a^{1}$, for $a>0$ and $x\in \mathbb{R}^3$. Below we recall such estimates.

\begin{lemma}[Gagliardo-Nirenberg type inequality \cite{CF20}, \cite{CF23}]\label{lemaradialGN1} Assume $0<b<3/2$. Then, for all $\eta>0$, there exists a constant $C_\eta>0$ such that 
\begin{itemize}
\item[(i)][Radial] for all $u\in \dot{H_a^{s_c}}(\Real^3)\cap \dot H_a^1(\Real^3)$ with radial symmetry the following inequality holds
	\begin{equation}\label{GNradial}
	\int_{|x|\geq R}|x|^{-b}|u|^{2\sigma+2}\,dx\leq \eta\|\nabla u\|_{L^{2}}^{2}+\frac{C_{\eta}}{R^{2(1-s_ c)}}\left\{[\rho(u,R)]^{\frac{2+\sigma}{2-\sigma}}+[\rho(u,R)]^{\sigma+1}\right\},
	\end{equation}
	if $\frac{2-b}{3}<\sigma<2-b$.
\item[(ii)][Non-radial] for all $u\in \dot{H_a^{s_c}}(\Real^3)\cap \dot H_a^1(\Real^3)$ the following inequality holds
	\begin{equation}\label{GNradial2}
		\int_{|x|\geq R}|x|^{-b}|u|^{2\sigma+2}\,dx\leq \eta\|\nabla u\|_{L^{2}}^{2}+\frac{C_{\eta}}{R^{2(1-s_ c)}}[\rho(u,R)]^{\frac{2-\sigma}{2-3\sigma}},
	\end{equation}
	if $\frac{2-b}{3}<\sigma<\min\{2-b, \frac{2}{3}\}$.
\end{itemize}
\end{lemma}


Next, we use the previous lemma to refine the estimate \eqref{estintradial}. Indeed, since $\rho$ non-increasing, for all $u\in \dot{H}_a^{s_c} \subset L^{\sigma_c}$ we have
	\begin{align}
		\int_{|x|\geq R} |x|^{-2}|u|^2\,dx&=\sum_{j=0}^\infty \int _{2^jR\leq |x|\leq 2^{j+1}R}|x|^{-2}|u|^2\,dx\leq 	\sum_{j=0}^\infty\dfrac{1}{(2^jR)^2}\int_{2^jR\leq |x|\leq 2^{j+1}R}|u|^{2}\,dx\\ 
		&\leq \frac{1}{R^{2(1-s_c)}}\rho(u,R)\sum_{j=0}^\infty \frac{1}{(2^{2(1-s_c)})^j}\leq \frac{C}{R^{2(1-s_c)}}\rho(u,R). \label{GNradial3}
	\end{align}
Define the number $\gamma=\gamma(\sigma, N)$ by
\begin{align}
	\displaystyle\gamma=\left\{
	\begin{array}{cc}
	\dfrac{2+\sigma}{2-\sigma},& \mbox{ if $\frac{2-b}{3}<\sigma<2-b$ and $u\in \dot{H}_a^{s_c}\cap\dot{H}_a^{1}$ is radial},\\
	&\\
	\dfrac{2-\sigma}{2-3\sigma },& \mbox{ if $\frac{2-b}{3}<\sigma<\min\{2-b, \frac{2}{3}\}$ and $u\in \dot{H}_a^{s_c}\cap\dot{H}_a^{1}$}.
	\end{array} 
\right.
\end{align} 
Note that in both cases, we have
\begin{align}
	\gamma>\sigma+1>1.
\end{align}
Thus, inserting \eqref{GNradial}, \eqref{GNradial2}, \eqref{GNradial3} in \eqref{estintradial} we deduce
	\begin{align}
	8\sigma s_c\|\mathcal L_a^{\frac12}
 u(t)\|^2_{L^2}+z''_R(t)&-16(\sigma s_c+1)E[u_0]\\\leq &\eta\|\mathcal L_a^{\frac12}
 u(t)\|_{L^2}^2+\frac{C_{\eta}}{R^{2(1-s_c)}}\left\{[\rho(u(t), R)]^{\gamma}+\rho(u(t),R)\right\}.\label{estint}
\end{align}
Now, assume that $v\in C([0,\tau_*]: \dot H_a^{s_c}\cap \dot H_a^1)$ and fix $\tau_0\in [0,\tau_*]$. Choosing $\eta>0$ we can absorb the term $\eta\|\mathcal L_a^{\frac12}
 v(t)\|_{L^2}^2$ in the left hand side of the above inequality, moreover integrating in $t$ from $0$ to $\tau$ and then in $\tau$ from $0$ to $\tau_0$, we obtain\footnote{See the proof of estimate (4.15) in \cite{CF20} for more details.}
\begin{align}
	4\sigma s_c\int_{0}^{\tau_0}(\tau_0-\tau)&\|\mathcal L_a^{\frac12}
 v(\tau)\|_{L^2}^2\,d\tau+z_R(\tau_0)\\
	\leq& z_R(0)+\tau_0\left(z'_R(0)+8\tau_0(\sigma s_c+1)E[v_0]\right)\\
	\quad &+\frac{1}{R^{2(1-s_c)}}\int_{0}^{\tau_0}\int_{0}^{\tau}\left\{[\rho(v(t), R)]^{\gamma}+\rho(v(t),R)\right\}\,dtd\tau\\\leq& cR^{2(1+s_c)}\|v_0\|_{L^{\sigma_c}}^2+\tau_0\left(z'_R(0)+8\tau_0(\sigma s_c+1)E[v_0]\right)\\
	\quad &+\frac{1}{R^{2(1-s_c)}}\int_{0}^{\tau_0}\int_{0}^{\tau}\left\{[\rho(v(t), R)]^{\gamma}+\rho(v(t),R)\right\}\,dtd\tau.\label{ineq1}
\end{align}
	where we have used, from inequality \eqref{radial1}, that
\begin{align}\label{phiR1}
	z_R(0)\leq cR^{2}\int_{|x|\leq 4R}|v_0|^2\,dx\leq cR^{2(1+s_c)}\|v_0\|_{L^{\sigma_c}}^2.
\end{align}
Next, observe that for all $R>0$ the definition of $\phi$ (see \eqref{phi}) yields
\begin{align}\label{zRt0}
	z_R(\tau_0)
	&\geq \int_{R\leq |x|\leq 2R}|x|^2|v(\tau_0)|^{2}\,dx\geq {R^2}\int_{R\leq|x|\leq 2R}|v(\tau_0)|^2\,dx.
\end{align}
Thereby, injecting \eqref{zRt0} into \eqref{ineq1} we have
\begin{align}
	4\sigma s_c\int_{0}^{\tau_0}(\tau_0-\tau)&\|\mathcal L_a^{\frac12}
 v(\tau)\|_{L^2}^2\,d\tau+{R^2}\int_{R\leq|x|\leq 2R}|v(\tau_0)|^2\,dx
	\\\leq& cR^{2(1+s_c)}\|v_0\|_{L^{\sigma_c}}^2+\tau_0\left(z'_R(0)+8\tau_0(\sigma s_c+1)E[v_0]\right)\\
	\quad &+\frac{1}{R^{2(1-s_c)}}\int_{0}^{\tau_0}\int_{0}^{\tau}\left\{[\rho(v(t),R)]^{\gamma}+\rho(v(t),R)\right\}\,dtd\tau.\label{ineq2}
\end{align}

\subsection{Preliminary lemmas}
We now recall two results that we will be used in the proof of our main theorems. The first one provides a uniform control of the $\rho$ norm and a global dispersive estimate. 

\begin{prop}\label{prop1}
Let $0<b<\frac{3}{2}$. Assume either $\frac{2-b}{3}<\sigma<2-b$ and $v_0\in \dot{H_a^{s_c}}(\Real^3)\cap \dot H_a^1(\Real^3)$ is radial or $\frac{2-b}{3}<\sigma<\min\{2-b, \frac{2}{3}\}$ and $v_0\in \dot{H}_a^{s_c}\cap\dot{H}_a^{1}$. Let $v\in C\left([0,\tau_\ast]: \dot{H_a^{s_c}}(\Real^3)\cap \dot H_a^1(\Real^3)\right)$ be a solution to \eqref{PVI} such that the initial data $v(0)=v_0$ satisfies
	\begin{equation}\label{hpprop1i}
	\tau_\ast^{1-s_c}\max\{E[v_0],0\}<1
	\end{equation}
	and
	\begin{align}\label{hpprop1ii}
	M_0:=\frac{4\|v_0\|_{L^{\sigma_c}}}{K_a}\geq 2,
	\end{align}
where $K_a$ is a optimal constant in \eqref{GNacrit2}. Then, there exist universal constants $C_1,\alpha_1,\alpha_2>0$ depending only on $N,\sigma$ and $b$  such that, for all $\tau_0\in [0,\tau_\ast]$, the following uniform estimates hold
	\begin{align}\label{prop1ii}
	\rho(v(\tau_0),M_0^{\alpha_1}\sqrt{\tau_0})\leq C_1M_0^2
	\end{align}
and 
	\begin{align}\label{prop1i}
	\int_0^{\tau_0}(\tau_{0}-\tau)\|\mathcal L_a^{\frac12}
 v(\tau)\|_{L^2}^2\,d\tau \leq M_0^{\alpha_2}\tau_{0}^{1+s_c}.
	\end{align}
\end{prop}

The proof of Proposition \ref{prop1} follows exactly the same outline as the proof of Proposition 3.2 in \cite{CF23}, just using the new estimate \eqref{ineq2} and the equivalence \eqref{equiv}. The next result shows that restricting the initial data in a neighborhood of the origin we can deduce a lower bound for its $L^2$ norm.

\begin{prop}\label{prop2} Under the same assumptions of Proposition \ref{prop1}, let 
	\begin{align}\label{hpprop2i}
	\tau_0\in\left[0,\frac{\tau_*}{2}\right].
	\end{align}
Define $\lambda_v(\tau)=\|\mathcal L_a^{\frac12}
 v(\tau)\|^{-\frac{1}{1-s_c}}_{L^2}$ and assume that
	\begin{align}\label{Ev}
	E[v_0]\leq \frac{\|\mathcal L_a^{\frac12}
 v(\tau_0)\|^2_{L^2}}{4}=\frac{1}{4\lambda_v^{2(1-s_c)}(\tau_0)}.
	\end{align}
Then, there exist universal  constants $C_2, \alpha_3>0$ depending only on $N,\sigma$ and $b$ such that if
	\begin{align}\label{F}
	F_*=\frac{\sqrt{\tau_0}}{\lambda_v(\tau_0)}\,\,\,\mbox{ and }\,\,\,D_*=M_0^{\alpha_3}\max[1,F_*^{\frac{1+s_c}{1-s_c}}],
	\end{align}
	then
	\begin{align}\label{tprop2}
	\frac{1}{\lambda_v^{2s_c}(\tau_0)}\int_{|x|\leq D_*\lambda_v(\tau_0)}|v_0|^2\,dx\geq C_2.
	\end{align}
\end{prop}
\begin{proof}

Given $v(\tau_0)$ we define
	\begin{align}
	w(x)=\lambda_v^{\frac{2-b}{2\sigma}}(\tau_0)v(\lambda_v(\tau_0)x,\tau_0),
	\end{align}
where $\lambda_v(\tau_0)=\|\mathcal L_a^{\frac12}
 v(\tau_0)\|^{-\frac{1}{1-s_c}}_{L^2}$,
 and
 \begin{align}\label{Aep}
J_{*}=C_{*}\max[M_0^{\alpha_1}F_*,M_0^{\frac{\gamma}{1-s_c}}],
\end{align}
for $C_*>1$ large enough.
Following the same arguments as in the proof of Proposition 3.3 in \cite{CF23}, there exists a constant $c>0$ such that
	\begin{align}\label{364}
	c\leq \int_{|y|\leq 2 J_*}|w|^2\,dx=\frac{1}{\lambda_v^{2s_c}(\tau_0)}\int_{|x|\leq 2J_*\lambda_v(\tau_0)}|v(\tau_0)|^2\,dx.
	\end{align}

Now, we want to verify that the above estimate is also verified for the initial data $v_0$. Indeed, for $R>0$, applying H\"older's inequality and equivalence \eqref{equiv} we have
\begin{align}
	\left|\frac{d}{d\tau}\int \varphi_{R}|v|^2\,dx\right|=2\left|\textit{Im}\left(\int \nabla \varphi_{{R}}\cdot \nabla v\overline{v}\,dx\right)\right|&\leq \frac{c}{{R}}\|\mathcal L_a^{\frac12}
 v(\tau)\|_{L^2}\left(\int_{{R}\leq |x|\leq 2{R}}|v(\tau)|^2\,dx\right)^{\frac{1}{2}}
	\\&\leq c\frac{\rho(v(\tau),{R})^{\frac{1}{2}}}{{{R}}^{1-s_c}}\|\mathcal L_a^{\frac12}
 v(\tau)\|_{L^2}.
\end{align}
Next, integrating the above inequality from $0$ to $\tau_0\in[0,\frac{\tau_*}{2}]$ and using Cauchy-Schwarz inequality and \eqref{prop1i} we get
\begin{align}
\left|\int \varphi_{{R}}|v(\tau_0)|^2\,dx\right.&\left.-\int \varphi_{{R}}|v_0|^2\,dx\right|\\
&\leq c\frac{\rho(v(\tau),{R})^{\frac{1}{2}}}{{R}^{1-s_c}}\int_0^{\tau_0}\|\mathcal L_a^{\frac12}v(\tau)\|_{L^2}\,d\tau \leq c\frac{\rho(v(\tau),{R})^{\frac{1}{2}}}{{R}^{1-s_c}}\left(\tau_0\int_0^{\tau_0}\|\mathcal L_a^{\frac12}
 v(\tau)\|_{L^2}^2\,d\tau\right)^{\frac12}\\
&\leq c\frac{\rho(v(\tau),{R})^{\frac{1}{2}}}{{R}^{1-s_c}}\left(\int_0^{2\tau_0}(2\tau_0-\tau)\|\mathcal L_a^{\frac12}
 v(\tau)\|_{L^2}^2\,d\tau\right)^{\frac{1}{2}} \leq c\rho(v(\tau),{R})^{\frac{1}{2}} M_0^{\frac{\alpha_2}{2}}{R}^{2s_c}\left(\frac{\sqrt{\tau_0}}{R}\right)^{1+s_c}.\label{prop21}
\end{align}
Given $0<\varepsilon\ll 1$, define
\begin{align}\label{Dep}
	D_\varepsilon=\frac{1}{\varepsilon} \max\left\{F_*,F_*^{\frac{1+s_c}{1-s_c}}\right\}\max\left\{M_0^{\alpha_1},M_0^{\frac{2+\alpha_2}{2(1-s_c)}}\right\},
\end{align}
and for $D\geq D_\varepsilon$ consider
\begin{align}\label{Rtil}
	\widetilde{R}=\widetilde{R}(D,\tau_0)=D\lambda_v(\tau_0).
\end{align}
From \eqref{F} and \eqref{Rtil}, we get 
\begin{align}
	\tilde{R}=D\lambda_v(\tau_0)\geq D\frac{\lambda_v(\tau_0)}{\sqrt{\tau_0}}\sqrt{\tau_0}\geq \frac{D}{F_*}\sqrt{\tau}
	, \,\,\,\mbox{for all}\,\,\, \tau\in [0,\tau_0],
\end{align}
and thus, using \eqref{Dep}, the monotonicity of $\rho$ and \eqref{prop1ii}, we have
\begin{align}
	\rho(v(\tau),\tilde{R})\leq \rho(v(\tau),M_0^{\alpha_1}\sqrt{\tau})\leq cM_0^2, \,\,\,\mbox{for all}\,\,\, \tau\in [0,\tau_0].
\end{align}
Thereby, from \eqref{prop21}
\begin{align}
	\frac{1}{\lambda^{2s_c}_v(\tau_0)}\left|\int \varphi_{\tilde{R}}|v(\tau_0)|^2\,dx-\int \varphi_{\tilde{R}}|v_0|^2\,dx
	\right|
	\leq cM_0^{1+\frac{\alpha_2}{2}}D^{2s_c}\left(\frac{F_*}{D}\right)^{1+s_c}=cM_0^{\frac{2+\alpha}{2}}\frac{F_*^{1+s_c}}{D^{1-s_c}}.
\end{align}
Finally, since $D\geq D_{\varepsilon}\geq \frac{1}{\varepsilon} F_*^{\frac{1+s_c}{1-s_c}}M_0^{\frac{2+\alpha_2}{2(1-s_c)}}$ by \eqref{Dep}, we obtain \begin{align}
	\frac{1}{\lambda^{2s_c}_v(\tau_0)}\left|\int \varphi_{\tilde{R}}|v(\tau_0)|^2\,dx-\int \varphi_{\tilde{R}}|v_0|^2\,dx
	\right|
	< \varepsilon
\end{align}
which concludes the proof of Proposition \ref{prop2}.

\end{proof}

\subsection{Existence of blow-up solutions and lower bound for the blow-up rate}
We are now in position to prove our main results. We start with the existence of blow-up solutions.

\begin{proof}[Proof of Theorem \ref{scteo1}] Let $u\in C([0,+\infty):\dot H_a^{s_c}\cap \dot H_a^1 )$ be a global solution with $E[u_0]\leq 0$. In view of the Gagliardo-Nirenberg type inequality \eqref{GNacrit2} and $a>0$, we deduce that
	\begin{align}
	0\geq E[u_0]\geq \frac{1}{2}\|\mathcal L_a^{\frac12} u_0\|_{L^2}^2\left[1-\left(\frac{\|u_0\|_{L^{\sigma_c}}}{K_a}\right)^{2\sigma}\right].
	\end{align}
and therefore the assumptions of Proposition \ref{prop1} are satisfied which implies 
	\begin{align}
	\int_{0}^{\tau_*}(\tau_*-\tau)\|\mathcal L_a^{\frac12} u(\tau)\|_{L^2}^{2}\,d\tau\leq C_{u_0}\tau_*^{1+s_c},\,\, \mbox{for all} \,\, \tau_*>0.
	\end{align}	
Define  $\lambda_u(\tau)=\|\mathcal L_a^{\frac12} u(\tau)\|^{-\frac{1}{1-s_c}}_{L^2}$. By previous inequality there exists a sequence $\tau_n\nearrow +\infty$ such that
\begin{align}\label{lambdatau}
	\lambda_u(\tau_n)\geq C_{u_0}\sqrt{\tau_n},\,\,\,\mbox{for all}\,\,\, n\in \mathbb{N},
\end{align}
and in particular $\lambda_u(\tau_n)\nearrow +\infty$. Therefore, since $u_0\in \dot H_a^{s_c} \subset L^{\sigma_c}$, the relation \eqref{radial2} implies
\begin{align}\label{contrteo11}
	\lim_{n\to +\infty}\frac{1}{(D\lambda_u(\tau_n))^{2s_c}}\int_{|x|\leq D\lambda_u(\tau_n)}|u_0|^2\,dx=0, \,\,\,\mbox{for all}\,\,\, D\geq 0.
	\end{align}

On the other hand, we can apply Proposition \ref{prop2} at time $\tau_n$, for all $n\in \mathbb{N}$. Indeed, since $F_n = \frac{\sqrt{\tau_n}}{\lambda_u(\tau_n)} \leq C_{u_0}$, there exists $D_*$, independent of $n$, such that 
	\begin{align}\label{contrteo1}
	\frac{1}{\lambda_u^{2s_c}(\tau_n)}\int_{|x|\leq D_*\lambda_u(\tau_n)}|u_0|^2\,dx\geq C_2>0,
	\end{align}
reaching a contradiction with \eqref{contrteo11}. Therefore the solution cannot exist for all positive time completing the proof.
\end{proof}

Next, we refine the analysis to show a lower bound for the blow-up rate of the critical Sobolev norm.
\begin{proof}[Proof of Theorem \ref{scteo2}]

Let $u_0\in \dot{H_a^{s_c}}\cap \dot H_a^1$ so that the maximal time of existence $T^{\ast}>0$ of the corresponding solution $u$ to \eqref{PVI} is finite. The inequality \eqref{BUA2} yields
\begin{align}\label{HYPLB}
		\|\mathcal L_a^\frac12 u(t)\|_{L^2}\geq \frac{c}{(T^*-t)^{\frac{1-s_c}{2}}}, \,\,\, \textit{for all} \,\,\, t\in [0,T^*).
	\end{align}
	
For $\sigma_c=\frac{6\sigma}{2-b}$, we shall prove that  
\begin{align}\label{rate2}
	 \|u(t)\|_{L^{\sigma_c}}\geq |\log (T^{\ast}-t)|^{\alpha},\,\,\,as\,\,\,t\to T^{\ast},
	\end{align}
for a universal constant $\alpha=\alpha(N,\sigma,b)>0$, which implies the same lower bound for the critical Sobolev norm $\dot{H}_a^{s_c}$. The proof follows the ideas introduced by Merle and Raph\"ael \cite{MR_Bsc} and is very similar to the proof of Theorem 1.2 in \cite{CF20}. Below we sketch the main steps. First, for $t$ close enough to $T^*$, define the renormalization of $u$ given by
\begin{align}\label{v}
	v^{(t)}(x,\tau)=\lambda_u^{\frac{2-b}{2\sigma}}(t)\bar{u}(\lambda_u(t)x,t-\lambda^2_u(t)\tau),
	\end{align}
	where $\lambda_u(t)=\|\mathcal L_a^{\frac12} u(t)\|^{-\frac{1}{1-s_c}}_{L^2}$. 
One can check that $v^{(t)}(\tau)\in C([0,\tau^{(t)}]:\dot H^{s_c}\cap \dot H^1)$ is a solution to \eqref{PVI} with $\tau^{(t)}={t}/{\lambda_u^2(t)}$ and also Propositions \ref{prop1} and \ref{prop2} can be applied to $v^{(t)}(\tau_0)$ for any $\tau_0\in\left[0,{1}/{\lambda_u(t)}\right]$ (see \cite[Lemma 6.1]{CF20}). From now on, to simplify the notation, we write $v=v^{(t)}$.

Define $N(t)=-\log \lambda_u(t)$. Since $\|u(t)\|_{L^{\sigma_c}}=\|v(0)\|_{L^{\sigma_c}}$, in view of inequality \eqref{HYPLB}, the problem is reduced to find $\alpha=\alpha(N,\sigma, b)>0$ such that, for $t$ close enough to $T^{\ast}$, we have
\begin{align}\label{N1}
	\|v(0)\|_{L^{\sigma_c}}\geq [N(t)]^{\alpha}.
	\end{align}
Recalling the constants $K_a$ and $\alpha_2$ given in \eqref{GNacrit2} and \eqref{prop1i}, respectively, we define 
	\begin{align}\label{Mt1}
	L(t)=[100[M(t)]^{\alpha_2}]^{\frac{1}{2(1-s_c)}}, \,\,\,\mbox{where}\,\,\, M(t)=\frac{4\|v(0)\|_{L^{\sigma_c}}}{K_a}\geq 2. 
	\end{align}
If $L(t)\geq e^{\frac{\sqrt{N(t)}}{2}}$, then it is easy to see that \eqref{N1} holds with $\alpha=1$. On the other hand, if $L(t)<e^{\frac{\sqrt{N(t)}}{2}}$, we pick times $\tau_i\in[0,e^i]$, for all $i\in [\sqrt{N(t)},N(t)]$, such that
	\begin{align}\label{Ftaui}
	F(\tau_i)\leq L(t)\,\,\,\,\mbox{ and }\,\,\,\,\,\frac{1}{10L(t)}e^{\frac{i-1}{2}}\leq \lambda_v(\tau_i)\leq \frac{10}{L(t)}e^{\frac{i}{2}},
	\end{align}
where $F(\tau)=\sqrt{\tau}\left/\right.\lambda_v(\tau)$ (see the proof of inequality (6.10) in \cite{CF20}). Moreover, for this choice of $\tau_i$, there exist universal positive constants $\alpha_4(N,\sigma, b)$ and $c_4(N,\sigma,b)$ such that
\begin{align}\label{teste}
	\int_{\mathcal{C}_i}|v(0)|^{\sigma_c}\,dx\geq \frac{c_4}{[M(t)]^{\alpha_4s_c\sigma_c}}
	\end{align}
	where
	\begin{align}\label{Rti1}
	\mathcal{C}_i=\left\{x\in \Real^N;\, \frac{\lambda_v(\tau_i)}{[M(t)]^{\alpha_4}}\leq |x|\leq [M(t)]^{\alpha_4}\lambda_v(\tau_i)\right\},
	\end{align}
(see the proof of inequality (6.12) in \cite{CF20}). 

To finish the proof, let  $p(t)>1$ an integer such that 
	\begin{align}\label{cotapt}
	10^3[M(t)]^{2\alpha_4}\leq e^{\frac{p(t)}{2}}\leq 10^7[M(t)]^{2\alpha_4}.
	\end{align}
If $p(t)\geq \sqrt{N(t)}$, then simple computations imply \eqref{N1} with $\alpha=1/4\alpha_4$. In the remaining case $p(t)<\sqrt{N(t)}$ consider $(i,i+p(t))\in [\sqrt{N(t)},N(t)]\times [\sqrt{N(t)},N(t)]$. From\eqref{Ftaui} and \eqref{cotapt} we get
	\begin{align}
\lambda_v(\tau_{i+p(t)})&\geq \frac{e^{\frac{i+p(t)-1}{2}}}{10L(t)}
	\geq\frac{10^3[M(t)]^{2\alpha_4}e^{\frac{i-1}{2}}}{10L(t)}\geq [M(t)]^{2\alpha_4}\frac{10e^{\frac{i}{2}}}{L(t)}\geq [M(t)]^{2\alpha_4}\lambda_v(\tau_i),
	\end{align}
which implies that the annuli $\mathcal{C}_i$ and $\mathcal{C}_{i+p(t)}$ given by \eqref{Rti1} are disjoints. In this case, for $t$ close enough to $T^{\ast}$, there are at least $\frac{N(t)}{10p(t)}$ disjoint annuli, thus \eqref{teste} yields
	\begin{align}
	\|v(0)\|_{L^{\sigma_c}}^{\sigma_c}\geq \sum_{k=0}^{\frac{N(t)}{10p(t)}}\int_{\mathcal{C}_{kp(t)+\sqrt{N(t)}}}|v(0)|^{\sigma_c}\,dx\geq \frac{c_4}{[M(t)]^{\alpha_4\sigma_cs_c}}\frac{\sqrt{N(t)}}{10}=C\frac{\sqrt{N(t)}}{\|v(0)\|_{L^{\sigma_c}}^{\alpha_4\sigma_cs_c}},
	\end{align}
which implies \eqref{N1}, finishing the proof of Theorem \ref{scteo2}.
\end{proof}
%
%

\vspace{0.5cm}
\noindent 
\textbf{Acknowledgments.} L.C. was partially supported by Coordena\c{c}\~ao de Aperfei\c{c}oamento de Pessoal de N\'ivel Superior - CAPES. M.C. was partially supported by Coordena\c{c}\~ao de Aperfei\c{c}oamento de Pessoal de N\'ivel Superior - CAPES and Funda\c{c}\~ao de Amparo a Pesquisa do Estado do Piau\'i - FAPEPI. L.G.F. was partially supported by Coordena\c{c}\~ao de Aperfei\c{c}oamento de Pessoal de N\'ivel Superior - CAPES, Conselho Nacional de Desenvolvimento Cient\'ifico e Tecnol\'ogico - CNPq and Funda\c{c}\~ao de Amparo a Pesquisa do Estado de Minas Gerais - FAPEMIG. 

\newcommand{\Addresses}{{
		\bigskip
		\footnotesize
		
		LUCCAS CAMPOS, \textsc{Department of Mathematics, UFMG, Brazil}\par\nopagebreak
		\textit{E-mail address:} \texttt{luccas@ufmg.br}
		
		\medskip
		
		MYKAEL A. CARDOSO, \textsc{Department of Mathematics, UFPI, Brazil}\par\nopagebreak
		\textit{E-mail address:} \texttt{mykael@ufpi.edu.br}
		
		\medskip
		
		LUIZ G. FARAH, \textsc{Department of Mathematics, UFMG, Brazil}\par\nopagebreak
		\textit{E-mail address:} \texttt{farah@mat.ufmg.br}

}}
\setlength{\parskip}{0pt}
\Addresses

\end{document}